\documentclass[12pt]{article}

\usepackage{times}

\def\math#1{$#1$}

\def\mand#1{$$#1$$}

\def\frac#1#2{{#1\over #2}}

\def\mld#1{\begin{equation}
#1
\end{equation}}

\def\eqan#1{\begin{eqnarray*}
#1
\end{eqnarray*}}

\DeclareSymbolFont{AMSb}{U}{msb}{m}{n}
\DeclareMathSymbol{\N}{\mathbin}{AMSb}{"4E}
\DeclareMathSymbol{\Z}{\mathbin}{AMSb}{"5A}
\DeclareMathSymbol{\R}{\mathbin}{AMSb}{"52}
\DeclareMathSymbol{\Q}{\mathbin}{AMSb}{"51}
\DeclareMathSymbol{\I}{\mathbin}{AMSb}{"49}
\DeclareMathSymbol{\C}{\mathbin}{AMSb}{"43}

\def\qed{\hfill\rule{2mm}{2mm}}

\def\floor#1{{\left\lfloor\,#1\,\right\rfloor}}

\def\r#1{{(\ref{#1})}}

\def\dotfil{\leaders\hbox to 1.5mm{.}\hfill}
\newcounter{rmnum}
\def\RN#1{\setcounter{rmnum}{#1}\uppercase\expandafter{\romannumeral\value{rmnum}}}
\def\rn#1{\setcounter{rmnum}{#1}\expandafter{\romannumeral\value{rmnum}}}

\usepackage{color,graphics,enumerate,wrapfig}
\usepackage{amsmath,array,graphics,color,epsfig}
\usepackage[noend]{algorithmic}

\newtheorem{theorem}{Theorem}
\newtheorem{lemma}{Lemma}

\newtheorem{corollary}{Corollary}

\newcommand{\cT}{{\cal T}}

\newcommand{\qedsymb}{\hfill{\rule{2mm}{2mm}}}
\providecommand\remove[2]{}
\newenvironment{proof}{\begin{trivlist}
\item[\hspace{\labelsep}{\bf\noindent Proof: }]}{\qedsymb\end{trivlist}}
\date{}
\title{\bf  Embedding a Forest in a Graph.}
\author
{Mark K. Goldberg  and Malik Magdon-Ismail\\
Department of Computer Science,\\
Rensselaer Polytechnic Institute \\
Troy, NY, 12180.\\
{goldberg@cs.rpi.edu}; \hspace*{0.1in} {magdon@cs.rpi.edu}\\
}

\begin{document}
\maketitle
\begin{abstract}
For \math{p\ge 1},
we prove that every forest with \math{p} trees whose sizes are 
$a_1, \ldots, a_p$ can be embedded in any graph containing
at least
$\sum_{i=1}^p (a_i + 1)$ vertices and having a minimum degree 
at least $\sum_{i=1}^p a_i$.
\end{abstract}

\section{Introduction.}
It is a folklore fact that 
every tree with $d\geq 0$ edges can be embedded 
in any  graph with minimum vertex degree $d$.  
Indeed, a linear algorithm to find such an embedding
would 
sequentially embed the vertices of the tree according to a 
depth first search ordering of the tree vertices.
It is likely, though, that the required bound on the minimum degree is
excessive, as captured by 
the famous conjecture by Erd\"{o}s and S\'{o}s 
(\cite{erdos}), which  states that every tree with $d$ edges can be embedded 
in any graph whose average degree is at least $d$.
A number of results (\cite{B-D, mclennan, sidor, w-l-l, S-W}) confirm 
the conjecture for some classes of trees and classes of graphs.  
The full conjecture is still neither proved, nor disproved.

A natural extension of the problem is to embed a forest 
in a graph.
If $F=\{T_1, \ldots, T_p\}$ is a forest of \math{p} trees whose sizes are
$a_1, \ldots, a_p$ respectively, then a necessary condition for embedding 
$F$ in a graph $G$ is that $|V(G)| \geq \sum_{i=1}^p (1 + a_i)$.
The straightforward tree embedding algorithm
outlined above may fail, 
even if the minimum
degree is at least $\sum_{i=1}^p a_i$.
However, we show that this condition on the minimum
degree (in addition to the obvious necessary condition)
is sufficient to guarantee that the forest can be embedded in 
the graph; we prove the following:
\begin{theorem}
\label{theoremforest:1}
Let $F = \{T_1, \ldots, T_p\}$  be a forest, and $d = \sum_{i=1}^p a_i$, 
where $a_i$ is the number  of edges in the tree $T_i$ $(i \in [1, p])$.
Then every graph $G$ with at least $d + p$ vertices and minimum degree
at least $d$ contains $F$ as a subgraph.
\end{theorem}
Our proof can be converted to a quadratic algorithm for embeding a forest.

We consider simple undirected graphs without parallel edges and loops. 
The set of 
vertices adjacent to a vertex $x$, the neighborhood of $x$, is denoted $N(x)$.
An embedding $f: H \rightarrow G$ 
of a graph $H$ in a graph $G$ is a one-to-one mapping 
\math{f : V(H) \rightarrow V(G)} such that for any two distinct vertices 
$x,y \in V(H)$, if $xy \in E(H)$ then $f(x)f(y) \in E(G)$. 
For a graph $H$, the order of $H$ is the number of its vertices (denoted $|H|$)
and the size of $H$ is the number of its edges.
For the terms not defined in this paper see (\cite{We}).

\section{A Proof of the Theorem \ref{theoremforest:1}}

We prove the theorem by induction on \math{p}, the number of trees in the
 forest.
We can assume that every tree in a forest has at least two vertices,
so \math{a_i\ge 1}.

\paragraph*{The Base Case, \math{p=1}.}
The forest in this case consists of a single tree $T_1$ with $d$ edges.
We prove a slightly stronger statement, which implies the theorem for
\math{p=1}. 
\begin{lemma}
\label{lemmatree:1}
Given a connected subgraph $C$ of $T_1$ and an embedding 
$f: C \rightarrow G$, there is an embedding \math{g: T_1 \rightarrow G} 
whose restriction to \math{C} is precisely \math{f}.
\end{lemma}

\begin{proof}
The idea is to arbitrarily grow the embedding \math{f} of 
\math{C} to an embedding \math{g} of \math{T_1}.
If \math{|C|<d+1}, let \math{uv \in E(T_1)} be an edge such that
\math{u\in V(C)} and \math{v\in V(T_1\setminus C)}.
Let \math{w = f(u)}. Since \math{C} has at most \math{d-1} vertices other 
than \math{u} and since the degree of \math{w} in \math{G} is at least
\math{d}, $G$ has an edge \math{wz} with vertex $z$ not in $g(C)$. Thus,
$f$ can be expanded to $g:C\cup\{v\} \rightarrow G$ by  defining 
$g (x) = g(x)$ for all $x \in C$, and $g(v) = z$. Iterating this expansion
completes the proof.
\end{proof}
\begin{corollary}
For any vertex $x$ of \math{T_1} and any vertex $y$ of $G$, an 
embedding $f: T_1 \rightarrow G$ exists for which $f(x) = y$. 
\end{corollary}

\paragraph*{The Induction Step, \math{p>1}.} 
Assuming the theorem holds for any forest $F_{p-1}$ 
with \math{p-1} trees, let \math{F_p} be a forest containing $p$
trees $T_1, \ldots, T_p$.  Denote $a_i$ the size of $T_i$ ($i \in [1,p]$).
Let $a_1 \geq a_2 \geq \ldots \geq a_p$, and let $a = a_1$.

\paragraph{Assumption.}
For the purpose of deriving a contradiction, we assume that 
\math{F_p} cannot be 
embedded in graph $G$ satisfying the conditions of the theorem.
\begin{lemma}
\label{lemmatree:2}
For every embedding \math{g: T_1\rightarrow G}, there is a vertex outside 
of \math{g(T_1)} which is adjacent to every vertex in \math{g(T_1)}.
\end{lemma}
\begin{proof}
If the statement were incorrect, then the removal of $g(T_1)$ from $G$ 
would leave a subgraph $G'$ with at least $d+p-(a+1)=\sum_{i=2}^p (1 + a_i) $ 
vertices each 
of degree at least $d-a\geq\sum_{i=2}^p a_i$. 
Inductively, $\{T_2, \ldots, T_p\}$ can be embedded in $G'$ 
which would yield an embedding of $F_p$ in $G$ contradicting
the assumption that $F_p$ cannot be embedded in $G$. 
\end{proof}
The main use of the previous lemma is to show that under 
our assumption, there is a large clique in \math{G}. 
\begin{lemma}
\label{lemmatree:3}
\math{G} contains a clique of size at least \math{a+2}.
\end{lemma}
\begin{proof}
Let $K$ be the largest clique in $G$ and let $|K| <a+2$. Select any connected 
subgraph  $C$ of $T_1$ of order $|C|= |K|$, and embed $C$ in $K$; this is
possible since $K$ is a clique. By Lemma \ref{lemmatree:1}, this embedding 
can be expanded to an embedding $f$ of $T_1$ in $G$, and by 
Lemma \ref{lemmatree:2} there is a vertex outside of $f(T_1)$ adjacent
to all vertices in $f(T_1)$. In particular, it is adjacent
to all vertices in $K$, contradicting $K$'s maximality. Thus, $|K| \geq a+2$. 
\end{proof}
It turns out that for the rest of the proof, we only need a clique of 
size \math{a}. 
\begin{lemma}
\label{lemmatree:4}
Any tree of order \math{a+1} can be embedded in any connected graph 
of order at least $a+1$ that contains a clique of order \math{a}.
\end{lemma}
\begin{proof}
Start by embedding a leaf at a vertex outside an \math{a}-clique, but adjacent
to a node in the clique (such a vertex must exist by connectivity). The 
remainder of the tree can be embedded in the clique.
\end{proof}

Let \math{K} be a clique of size $a$ in $G$. 
The subgraph $G' = G \setminus K$ contains at least $d-a+p$ 
vertices each of degree at least $d-a$. 
Inductively, $F_{p-1} = \{T_2, \ldots, T_p\}$ 
can be embedded in $G'$. Let $g: F_{p-1} \rightarrow G'$ be such an embedding. 
Select any vertex $x \in K$ and a subset $X \subseteq N(x) \setminus K$ with 
$|X| = d-a+1$ vertices. It is possible since $|N(x) \setminus K| \geq d-a+1$.
\begin{lemma}\label{lemma}
Every vertex in $X$ is  used  by any embedding \math{g} of $F_{p-1}$.
\end{lemma}
\begin{proof}
Indeed, if $x\in X \setminus g(T_{p-1})$ is not used, then 
by Lemma \ref{lemmatree:4}, $T_1$ can be embedded in the subgraph $H$ 
induced by $K \cup\{x\}$, which would yield an embedding of $F_p$.
\end{proof}

Since all $d-a+1$ vertices of \math{X} are used in the embedding
$g:F_{p-1}\rightarrow G$, exactly \math{p-2} vertices outside of 
\math{K\cup X}, denoted \math{y_1,\ldots,y_{p-2}},  are 
used by $g$. The remaining $m+1$ vertices of the graph, outside of
$K \cup g(T_{p-1})$, are denoted \math{s_0, s_1,\ldots,s_m}.
We now split the set of the trees of the forest 
\math{F_{p-1}} into four subsets $\cT_1, \cT_2, \cT_3,$ and $\cT_4$.

\noindent\hspace*{0.2in}$\cT_1$: 
trees which are embedded entirely in \math{X}; 
\\
\noindent\hspace*{0.2in}$\cT_2$:
trees whose embedding has at least two vertices in \math{X} and 
at least one vertex in \math{Y}; 
\\
\noindent\hspace*{0.2in}$\cT_3$:
trees whose embedding has only one vertex in \math{X};  and
\\
\noindent\hspace*{0.2in}$\cT_4$: 
trees whose embedding is entirely in \math{Y}. 

\noindent
Let $q_i = |\cT_i|$ ($i =1,2,3,4$).
Since every tree in \math{F_{p-1}} belongs to one of these four subsets,
\mand{
q_1 + q_2 + q_3 + q_4 = p-1.
}
Denote by $a(T_i)$ the size of $T_i$.
For the embedding \math{g}:
every tree in $\cT_2$ uses at 
least one vertex in $Y$;
and, every tree \math{T} in $\cT_3$ (resp. $\cT_4$) uses 
\math{a(T)} (resp. \math{1+ a(T)}) vertices in \math{Y}. Since
there are \math{p-2} vertices in \math{Y},
\eqan{
q_2 + \sum_{T_i\in \cT_3} a(T_i)+ \sum_{T_i \in \cT_4}(a(T_i) + 1)&\le& p-2 = 
q_1 + q_2 +q_3 +q_4 -1.
}
This immediately gives a lower bound for \math{q_1}.
\begin{lemma}
\label{lemmatree:5}
\math{q_1\ge1+\sum_{T_i\in\cT_3}(a(T_i)-1)+\sum_{T_i\in\cT_4}a(T_i)\ge1+q_4}.
\end{lemma}
Let $s$ be an arbitrary vertex in $S$. Our goal now is to evaluate the degree
of $s$ in the subgraph induced on $S$, based on the assumption that \math{F_p}
cannot be embedded.  
We start with
\mld{
|N(s) \cap S| \geq d - |N(s) \cap K|  - |N(s) \cap (X \cup Y)|.
\label{eq:1}
}
We make the following observations about the 
neighborhood of \math{s} in $K \cup X \cup Y$.
\begin{enumerate}[1.]
\item \math{s} is not adjacent to any vertex in \math{K}, else  by 
Lemma \ref{lemmatree:4}, \math{T_1} could be embedded in \math{s\cup K}.
\item
\math{s} is not adjacent to at least one  vertex in $g(T)$ for any tree
$T \in \cT_2 \cup \cT_3$. Indeed, if \math{s} is adjacent to 
every vertex in \math{T}, a vertex 
$g(T_i)$ which is in $X$ can be swapped with $s$; this gives
an embedding of \math{F_{p-1}} that doesn't use every
vertex of \math{X}, contradicting
Lemma \ref{lemma}.
\item \math{s} is not adjacent to at least two vertices of 
$g(T)$ for any tree $T\in\cT_1$. 
Indeed, let $s$ be adjacent to all but one vertex in $g(T)$, and let
$y = g(x)$ be that exceptional vertex. Then for every neighbor $x'$ 
(in \math{T})
of $x$,
$s$ is adjacent to $g(x')$. By setting $g(x) = s$, we obtain a valid embedding
of \math{F_{p-1}} 
which doesn't use a vertex in $X$, contradicting Lemma \ref{lemma}.
\end{enumerate}
So, \math{N(s)\cap K=\emptyset} and 
\math{N(s)\cap(X\cup Y)\le |X\cup Y|-(2q_1 + q_2 +q_3)}.
Since \math{|X\cup Y|=d-a+p-1}, we have from Inequality \r{eq:1} that 
the number of neighbors of \math{s} in \math{S} is at least:
\begin{eqnarray*}
|N(s) \cap S|&\ge &d-(d-a +p-1) +2q_1 +q_2 +q_3\\
&=&a+q_1-q_4\\
&\ge&a+1,
\end{eqnarray*}
where we have used \math{q_1+q_2+q_3+q_4=p-1} and Lemma \ref{lemmatree:5}.
Thus, the degree of any vertex~$s$ in the subgraph induced by $S$
is at least $a+1$. By Lemma \ref{lemmatree:1}, $T_1$ can be embedded in this
subgraph,
contradicting the Assumption, and completing the proof of
Theorem \ref{theoremforest:1}. \qed

\section{Conjecture}
When the number of vertices equals the lower bound \math{p+d} 
and the minimum degree is at least \math{d}, then the Hajnal-Szemer\'{e}di
theorem on equitable coloring \cite{hajnal, kostoch}, applied to the complement
of the graph, guarantees the existence of \math{p} cliques each of size at 
least \math{\floor{d/p}}.
Thus, an arbitrary \math{p} graphs of order at most \math{\floor{d/p}}
can be \emph{simultaneously} embedded in the graph. 
When the number of vertices
increases, however, cliques are no-longer guaranteed. Our
result shows that one can simultaneously embed
trees, even as the number of vertices grows, as long as the sum of the
tree sizes is at most \math{d}.

Alternatively, one can ask whether a bound on the minimum degree is
excessive to guarantee the embedability of a forest. Indeed, 
we propose a natural extension to the conjecture 
by Erd\"{o}s and S\'{o}s:
\begin{quote}
Let $F = \{T_1, \ldots, T_p\}$  be a forest, and $d = \sum_{i=1}^p a_i$,
where $a_i$ is the number  of edges in the tree $T_i$ $(i \in [1, p])$.
Then every graph $G$ with at least $d + p$ vertices and the average degree
$\geq d$ contains a subgraph isomorphic to $F$.
\end{quote}
For a single star, the conjecture clearly holds; but, even the
extension to a collection of stars is not clear.

\bibliographystyle{abbrv}
\bibliography{pap}
\end{document}